\documentclass[reqno]{amsart}
\usepackage{amsmath,amssymb,dsfont}
\usepackage{graphics,epsfig,color}
\usepackage[mathscr]{euscript}
\theoremstyle{plain}
\newtheorem{theorem}{Theorem}[section]
\newtheorem{proposition}[theorem]{Proposition}
\newtheorem{lemma}[theorem]{Lemma}

\newtheorem{claim}[theorem]{Claim}
\newtheorem{definition}[theorem]{Definition}

\newtheorem{remark}[theorem]{Remark}

\numberwithin{equation}{section}
\numberwithin{theorem}{section}
\newcommand{\mc}[1]{{\mathcal #1}}
\newcommand{\bb}[1]{{\mathbb #1}}

\newcommand{\id}{{1 \mskip -5mu {\rm I}}}
\renewcommand{\epsilon}{\varepsilon}
\renewcommand{\tilde}{\widetilde}
\renewcommand{\hat}{\widehat}

\renewcommand{\div}{\mathop{\rm div}\nolimits}

\definecolor{light}{gray}{.9}

\title[From level 2.5 to level 2 LD]{From level $2.5$ to level $2$ large deviations
for continuous time Markov chains}

\author[L.\ Bertini]{Lorenzo Bertini}
\address{Lorenzo Bertini \hfill\break \indent
   Dipartimento di Matematica, Universit\`a di Roma `La Sapienza'
   \hfill\break \indent
   P.le Aldo Moro 2, 00185 Roma, Italy}
 \email{bertini@mat.uniroma1.it}

\author[A.\ Faggionato]{Alessandra Faggionato}
\address{Alessandra Faggionato \hfill\break \indent
  Dipartimento di Matematica, Universit\`a di Roma `La Sapienza'
  \hfill\break \indent
  P.le Aldo Moro 2, 00185 Roma, Italy}
\email{faggiona@mat.uniroma1.it}

\author[D.\ Gabrielli]{Davide Gabrielli}
\address{Davide Gabrielli \hfill\break \indent
  Dipartimento di Matematica, Universit\`a dell'Aquila
  \hfill\break\indent
  Via Vetoio,   67100 Coppito, L'Aquila, Italy}
\email{gabriell@univaq.it}

\begin{document}

\begin{abstract}
We recover  the
Donsker--Varadhan large deviations principle (LDP)  for the empirical
measure  of  a continuous time Markov chain on a countable (finite or infinite) state
space from the  joint LDP for the empirical
measure and the empirical flow
 proved in \cite{BFG}.

\medskip


\medskip

\end{abstract}

\maketitle

\section{Introduction}
We consider a continuous time Markov chain $(\xi_t)_{t \geq 0}$ on a countable (finite or infinite) state space $V$. Following
\cite{N} the dynamics is defined knowing the \emph{jump rates} $r(x,y)$, $x \not
=y$ in $V$, under the assumption  that   $r(x):= \sum _{y \in V} r(x,y)<+\infty$ for all $x \in V$.  Then,   at each site $x$ the system waits an exponential time of parameter $r(x)$ afterwards it jumps to a state $y$ with probability $r(x,y)/r(x)$.
We  assume that a.s. for any fixed initial state explosion does not occur, hence the Markov chain is defined in $V$  for all times $t \in \bb R_+$  and we do not need to introduce any coffin state.    We denote  by $\bb P_x$ the law on the Skohorod space $D(\bb R_+; V)$ of the Markov chain starting at $x$.

In what follows we restrict to irreducible Markov chains such that there exists a unique invariant probability measure, which we denote by
  $\pi$.
  As in \cite{N}, by invariant probability measure $\pi$ we mean a
probability measure on $V$ such that
\begin{equation}
  \label{invariante}
  \sum _{y \in V} \pi(x) \, r(x,y)
  = \sum _{y \in V} \pi(y) \, r(y,x)\qquad \forall \:x \in V
\end{equation}
where we understand $r(x,x)=0$.
   We stress that the existence  of $\pi$ is guaranteed if $V$ is finite, while in general    uniqueness is automatic if  $\pi$ exists.

A fundamental result in the theory of large deviations is given by the Donsker--Varadhan Large Deviations Principle (LDP) of the empirical measure of Markov processes. We recall its statement referring to the above Markov chain $\xi$. Denote by $\mathcal{P}(V)$ the space of probability measures on $V$ endowed of the weak topology. Given $T>0$ the \emph{empirical measure}
$\mu_T\colon D(\bb R_+;V)\to \mc P(V)$ is defined by
\begin{equation}
  \mu_T \, (X) = \frac 1T\int_0^T\!dt \, \delta_{X_t}
\end{equation}
where $\delta_y$ denotes the pointmass at $y$.
Given $x\in V$, the ergodic theorem \cite{N} implies that the
empirical measure $\mu_T$ converges $\bb P_x$ a.s.\ to $\pi$ as $T \to
\infty$.
In particular, the family of probabilities $\{\bb P_x \circ
\mu_T^{-1}\}_{T>0}$ on $\mc P(V)$ converges to $\delta_\pi$.  In \cite{DV4} the large deviations from the above limit theorem have been studied by Donsker and Varadhan. Under suitable hypotheses (see Remark \ref{soleluna} below)
they  proved that   as $T\to+\infty$ the family of probability measures $\{\bb
  P_x\circ\mu_T ^{-1}\}_{T>0}$ on $\mc P(V)$
  satisfies a LDP with good rate function $\mc I$ such that
    \begin{equation}
  \mathcal{I}(\mu)=   \sup_h \left \{- \langle \mu,  h^{-1} Lh   \rangle \right \}
  \end{equation}
  as $h$ varies among the strictly positive functions in the domain of the infinitesimal generator $L$ (in general, $< \mu, f>:= \sum _{x \in V} \mu(x) f(x)$).

The above result has been derived in \cite{DV4}--(I) from an analogous result for discrete time Markov chains by an approximation argument in the case of $V$ finite. The extension to $V$ infinite has been achieved in \cite{DV4}--(III), while in \cite{DV4}--(IV) the  LDP for the empirical measure is obtained by contraction from the LDP for the empirical process.

 \medskip

Our aim in this note is to give an alternative proof of the LDP for the empirical measure  by contraction from the joint  LDP for the  empirical measure and flow recently proved in \cite{BFG}. The result is interesting since we obtain two different representations of the rate functional for the empirical measure, one in terms of a supremum and
the other one in terms of an infimum. As can been seen for example in the proof of Proposition \ref{simmetria}
this fact is very useful. Moreover the proof of our result is interesting by herself since it exploits several
discrete geometric feature of the underlying graph. We show also that in the finite dimensional case the coincidence
of the supremum with the infimum is an instance of the Fenchel-Rockafellar duality.

\smallskip

We recall the joint LDP for the empirical measure and flu in \cite{BFG} and fix some notation.
We denote  by  $E$ the  set of ordered edges in $V$ with positive transition rate, namely
  $ E := \{(y,z)\in V \times V: y
\not = z \text{ and } r(y,z)>0\}$. Then   for each $T>0$ we  define  the \emph{empirical flow} as the map $Q_T
\colon D(\bb R_+;V)\to [0,+\infty]^E$ given by
\begin{equation}
  \label{montecarlo}
  Q_T(y,z) \, (X) :=
  \frac{1}{T} \sum_{0\leq t\leq  T} \delta_y(X_{t^-}) \delta_z(X_{t})
  \qquad (y,z)\in E\,.
\end{equation}
Namely, $ T Q_T(y,z)$ is $\bb P_x$ a.s.\ the  number of jumps from
$y$ to $z$ in the time interval $[0,T]$ of the Markov chain $\xi $ starting at $x$. As discussed in \cite{BFG}, $Q_T(y,z) $ converges to $\pi(y) r(y,z)$ at $T \to \infty$ $\bb P_x$ a.s.

\smallskip

Elements in $[0, +\infty]^E$ are called \emph{flows}.  We denote by $L^1_+(E)$ the subset  of
 \emph{summable flows}, i.e. of flows $Q$ such that $\|Q\|_1 := \sum _{(y,z) \in E} Q(y,z) <+ \infty$.
 Given a summable  flow $Q\in L^1_+ (E) $    its
\emph{divergence} $\div Q \colon V\to \bb R$ is  defined
as
\begin{equation}
  \label{divergenza_fluss}
  \div Q \, (y)= \sum _{z : \, (y,z)\in E} Q(y,z)- \sum_{z:\, (z,y)\in E} Q(z,y),
  \qquad y\in V.
\end{equation}
 Observe that the divergence maps $L^1_+(E)$ into $L^1(V)$.

To  each probability $\mu \in \mc P(V)$
we associate the flow $Q^\mu\in \bb R_+^E$ defined by
\begin{equation}
  \label{Qmu}
  Q^\mu(y,z) := \mu(y) \, r(y,z)
  \qquad (y,z)\in E.
\end{equation}
Note that  $Q^\mu\in L^1_+(E) $ if and only if $\langle\mu,
r\rangle<+\infty$. Moreover, in this case, by \eqref{invariante}
$Q^\mu$ has vanishing divergence if only if $\mu$ is invariant for
the Markov chain $\xi$, i.e. $\mu=\pi$.

We endow $L^1_+(E)$  of the bounded weak* topology. As discussed in \cite{BFG} this topology  is the most suited  for
studying large deviations of the empirical flow. For completeness we recall its definition although it will never be used below (see \cite{Me} for a detailed treatment). A subset $W \subset L^1_+(E)$ is open if and only if for each $\ell >0$ the
set $\{ Q \in W\,:\, \|Q\|_1 <  \ell\}$ is open in the ball  $\{ Q \in L^1_+(E)\,:\, \|Q\|_1 <  \ell\}$ endowed of the weak*  topology inherited from $L^1(E)$. When $E$ is finite, the bounded weak* topology coincides with the $L^1$--topology.

\medskip

We can now recall the LDP proved in \cite{BFG}. We start from the assumptions.
To this aim,
 given $f\colon V\to \bb R$ such that  $\sum_{y\in V }r(x,y) \, |f(y)| <+\infty$
for each $x\in V$,
 we
denote by $Lf \colon V\to \bb R$ the function defined by
\begin{equation}
  \label{Lf}
  L f\, (x) := \sum_{y\in V} r(x,y) \big[ f(y)-f(x)\big]
  ,\qquad x\in V.
\end{equation}
\begin{definition}
 Given $\sigma \in \bb R_+$ we say that Condition $C(\sigma)$ holds if
  there exists a sequence of functions $u_n\colon V \to (0,+\infty)$
  satisfying the following requirements:
  \begin{itemize}
  \item [(i)] For each $x\in V$ and $n\in\bb N$ it holds $\sum_{y\in V} r(x,y)
    u_n(y) <+\infty$. 
  \item [(ii)] The sequence $u_n$ is uniformly bounded from below.
    Namely, there exists $c>0$ such that $u_n(x)\ge c$ for any $x\in
    V$ and $n\in\bb N$.
  \item[(iii)] The sequence $u_n$ is uniformly bounded from above on compacts.
    Namely, for each $x\in V$ there exists a
    constant $C_x$ such that for any $n\in\bb N$ it holds $u_n(x)\le C_x$.
  \item [(iv)] Set $v_n :=  - Lu_n / u_n$. The sequence
    $v_n\colon V\to \bb R$ converges pointwise to some $v\colon V\to \bb R$.
  \item[(v)] The function $v$ has compact level sets.
    Namely, for each $\ell\in \bb R$
    the level set $\big\{x \in V \,:\, v(x)\leq \ell\big\}$ is
    finite.
  \item[(vi)]
    There exists 
     a positive
    constant $C$ such that
    $v \ge \sigma \, r - C$.
  \end{itemize}
\end{definition}

 Let
$\Phi\colon \bb R_+ \times \bb R_+ \to [0,+\infty]$ be the function
defined by
\begin{equation}
  \label{Phi}
   \Phi (q,p)
   :=
   \begin{cases}
     \displaystyle{ q \log \frac qp - (q-p)}
     & \textrm{if $q,p\in (0,+\infty)$}
     \\
     \;p  & \textrm{if $q=0$, $p\in [0,+\infty)$}\\
     \; +\infty & \textrm{if $p=0$ and $q\in (0,+\infty)$.}
   \end{cases}
\end{equation}
For $p>0$, $\Phi( \cdot, p)$ is a nonnegative convex function and is zero only at $q=p$. Indeed, it  is the rate function for the LDP of the sequence  $N_T/T$ as $T \to +\infty$, $(N_t)_{t \in \bb R_+}$
being a Poisson process with parameter $p$.

 Finally,
we let $I\colon \mc P(V)\times L^1_+(E) \to
[0,+\infty]$ be the functional defined by
\begin{equation}
  \label{rfq}
  I (\mu,Q) :=
  \begin{cases}
    \displaystyle{
    \sum_{(y,z)\in E} \Phi \big( Q(y,z),Q^\mu(y,z) \big)
    }& \textrm{if  } \; \div Q =0\,,\; \langle \mu,r \rangle < +\infty
    \\
    \; +\infty  & \textrm{otherwise}.
  \end{cases}
\end{equation}
\begin{remark}\label{silente} As proved in \cite{BFG}[Appendix B]    the above condition
 $\langle \mu, r \rangle < +\infty$ can be removed, since
 the series  in \eqref{rfq} diverges  if $\langle \mu , r \rangle
 =+\infty$.
\end{remark}

\begin{remark}\label{soleluna} Condition $C(0)$ (i.e. $C(\sigma)$ with $\sigma=0$) with (i) replaced by the fact that $u_n $ belongs to the domain of the infinitesimal generator, and with $Lu_n$ defined as the infinitesimal generator applied to $u_n$, is the condition under which the large deviation of the empirical measure is derived in \cite{DV4}--(IV).
\end{remark}

\begin{theorem}[Bertini, Faggionato, Gabrielli, \cite{BFG}]$\,$
  \label{LDP:misura+flusso} \\Assume Condition $C(\sigma)$ to hold with $\sigma >0$.  (Alternatively, assume the hypercontractivity Condition 2.3 in \cite{BFG}).
  Then as $T\to+\infty$ the family of probability measures $\{\bb
  P_x\circ (\mu_T,Q_T)^{-1}\}$ on $\mc P(V)\times L^1_+(E)$
  satisfies a LDP with good and convex rate function $I$.
  Namely, for   each closed set $\mc C\subset \mc P(V)\times L^1_+(E)$,
  and each open  set $\mc A \subset \mc P(V)\times L^1_+(E)$, it
  holds for each $x \in V$
  \begin{align}
    \label{ubldp}
    & \varlimsup_{T\to+\infty}\;
    \frac 1T \log \bb  P_x \Big( (\mu_T,Q_T) \in \mc C \Big)
    \le -\inf_{(\mu,Q)\in \mc C} I(\mu,Q),
    \\
    \label{lbldp}
    & \varliminf_{T\to+\infty}\;
    \frac 1T \log \bb P_x \Big( (\mu_T,Q_T) \in \mc A \Big)
    \ge -\inf_{(\mu,Q)\in \mc A} I(\mu,Q).
   \end{align}
\end{theorem}
\begin{remark}
Condition $C(\sigma)$ with $\sigma>0$ (or alternatively the hypercontractivity Condition 2.3 in \cite{BFG})
implies that $\langle\pi ,r\rangle<+\infty$ (see Lemma 3.9 in \cite{BFG}). 
\end{remark}
\medskip

We can finally state our new results.
\begin{theorem} \label{teoteo}
Assume that
 Condition $C(\sigma)$  holds with $\sigma >0$   (alternatively, assume   the
  hypercontractivity Condition 2.3  in \cite{BFG}).
 Then
  as $T\to+\infty$ the family of probability measures $\{\bb
  P_x\circ\mu_T ^{-1}\}$ on $\mc P(V)$
  satisfies a LDP with good rate function $\mathcal{I} $ such that
 \begin{equation}\label{rosetta}
 \mathcal{I}(\mu)=   \inf\big\{ I(\mu,Q)\,:\: Q \in L^1_+(E) \big\}\,.
\end{equation}
$\mathcal{I}(\mu)<+\infty$ if and only if $\langle \mu, r \rangle < +\infty$, in this case
  the above infimum is indeed attained at a unique flow  $Q^* \in L^1_+(E)$.
Moreover the following alternative  variational characterization holds
  \begin{equation}
  \mathcal{I}(\mu)=
  \begin{cases}
    \sup\big\{ - \langle \mu,  e^{-g} L e^{g} \rangle \,: \,
 g\in L^\infty(V) \big\} &  \qquad \textrm{if}\ \langle \mu, r\rangle<+\infty\,,\\
  +\infty &\qquad  \textrm{otherwise}\,.
 \end{cases}
  \label{freddo}
\end{equation}
\end{theorem}

Since the projection map $\mathcal{P}(V) \times L^1_+(E) \ni (\mu,Q) \to \mu\in \mathcal{P}(V) $ is trivially continuous,  due to the contraction principle the first part of the theorem up to \eqref{rosetta}  follows from
Theorem  \ref{LDP:misura+flusso}.  Since $I(\mu,Q)=+\infty$ if $\langle \mu, r\rangle =+\infty$ and
$I(\mu,0)<+\infty$ if $\langle \mu, r \rangle <+\infty$,  we get that $\mathcal{I} (\mu)$ is finite if and only if $\langle \mu,r \rangle$ is finite. Finally,  note that since $I(\cdot, \cdot)$ is good, then the map $L^1_+(E) \ni Q \to  I(\mu,Q)$ is lower semicontinuous with compact level sets and therefore it has a minimum. The uniqueness of the minimizer follows from the fact that $I (\mu, \cdot)$ is strictly convex on the set $\{ Q\,:\; I(\mu, Q)<+\infty\}$, as can be easily derived from the strictly convexity
of $\Phi( \cdot, p)$ for $p>0$.
The non trivial task is  therefore to prove \eqref{freddo}.

We will use the following characterization of $I$ proved in \cite{BFG} (see formula (5.4) there)
\begin{equation}\label{cvi}
I(\mu,Q)= \sup I_{\phi,F} (\mu,Q)\,.
\end{equation}
In \eqref{cvi} the supreprum is among all pair $\phi,
  F$ with $\phi\in L^\infty(V)$, $F \in L^\infty(E)$, being respectively
$L^\infty(V)$ the set of bounded functions on vertices and
$L^\infty(E)$ the set of bounded functions on edges.
Moreover we have
\begin{equation}
 I_{\phi,F} (\mu,Q) :=
  \langle\phi,\div Q\rangle - \langle\mu,r^F-r\rangle+\sum_{(y,z)\in E}Q(y,z)F(y,z)
\end{equation}
where $r^F\colon V \to (0,+\infty)$ is defined by $r^F(y)
=\sum_{z\in V} r(y,z)e^{F(y,z)}$ and $\langle\phi,\div
Q\rangle=\sum_{y\in V} \phi(y)\div Q(y) $.
In \cite{BFG} formula \eqref{cvi} is proved with a slightly different
 class of functions but the argument can be clearly adapted to the present setting.
 See also Section \ref{rocky} for computations similar to \eqref{cvi}.

\medskip

In the reversible case we have the following additional result:
\begin{proposition}\label{simmetria} Assume the same setting of Theorem \ref{teoteo}.  Suppose that
the invariant measure $\pi$ is also reversible, i.e.
$
\pi(y)r(y,z)=\pi(z)r(z,y)
$ for all $y,z \in V$.
Suppose that $\mu \in \mathcal{P}(V)$ is such that $\mathcal{I}(\mu)<+\infty$ (i.e.\ $\langle \mu, r\rangle < +\infty$).
Then
$$
Q^*(y,z)=Q^*(z,y)=\sqrt{\mu(y)\mu(z)r(y,z)r(z,y)}
$$
is the minimizing flow in  \eqref{rosetta} and it holds
$$
\mc I (\mu)=\frac{1}{2}\sum_{y \in V}\sum_{z \in V}\left(\sqrt{\mu(y)r(y,z)}-\sqrt{\mu(z)r(z,y)}\right)^2\,.
$$
Moreover \eqref{freddo} admits a maximizing sequence $g^{(n)}$ that is a suitable approximating sequence
in $L^\infty(V)$ of the extended function $g:V\to \left\{-\infty\right\}\cup \mathbb R$
defined by
\begin{equation}
g(y):=\log \sqrt{\mu(y)/\pi(y) }\,.
\label{extended}
\end{equation}
\end{proposition}

The rest of the paper  is devoted to the proof of Theorem  \ref{teoteo} (see Section \ref{demo}) and the proof of Proposition \ref{simmetria} (see Section \ref{demo_simm}).  Most of the technical difficulties come from the case of $V$ infinite.   In Section \ref{rocky} we give for $V$ finite an alternative proof of Theorem \ref{teoteo} showing
that it is indeed a special case of the Fenchel-Rockafellar
Theorem. In the case $|V|<+\infty$ different proofs where given in \cite{WK} and \cite{BP}.


\section{Proof of Theorem \ref{teoteo}} \label{demo}

\subsection{Some preliminary results on oriented graphs}

Let $ (\mc V,\mc E)$ be an oriented graph. Given $y,z \in \mc V$, an \emph{oriented   path } from $y $ to $z$ in $ (\mc V,\mc E)$ is a finite  string  $(x_1,x_2, \dots, x_n)$  with $x_1=y$, $x_n = z$ and $(x_i,x_{i+1} ) \in \mc E$ for all $i=1, \dots, n-1$.
A \emph{cycle} in $(\mc V, \mc E)$ is an oriented path $(x_1,x_2, \dots, x_n)$ with $x_1=x_n$. It is called \emph{self--avoiding} if $x_i \not =x_j$ for $1\leq i <j <n $. Given a cycle $C$ we denote by $\id _C$ the function on $\mc E$ taking value $1$ on the edges $(x_i, x_{i+1})$, $1\leq i <n$, and zero otherwise.

Let us now refer to the oriented graph $(V,E)$.  We denote by $\mc C$ the family of self--avoiding cycles in $(V,E)$. Given $C \in \mc C$,
 note that $\id_C$ is a divergence--free flow in $L^1_+(E)$.
 In \cite{BFG}[Lemma 4.1] it is proved that any divergence--free flow $Q \in L^1_+(E)$  can be written as
$ Q = \sum _{ C \in \mc C} \hat Q (C) \id _C $ for suitable nonnegative constants $\hat Q (C)$, $C \in \mc C$. The above decomposition has to be thought as  $Q(y,z) =\sum _{ C \in \mc C} \hat Q (C) \id _C (y,z)$ for each edge $(y,z) \in E$.

Take $\mu \in \mathcal{P}(V)$ such that $\langle \mu, r \rangle < +\infty$. Consider the oriented graph $( V_\mu, E_\mu)$ where
\begin{align*}
& E_\mu = \{(y,z) \,:\, Q^\mu(y,z)= \mu(y) r(y,z) >0\}\,,\\
& V_\mu= \{ y \in V\,:\, \exists z \in V \text{ with } (y,z) \in E_\mu \text{ or } (z,y) \in E_\mu\}\,.
\end{align*}
Trivially, the support of $\mu$ is included in $V_\mu$. If $ z \in V_\mu \setminus {\rm supp} (\mu)$  then there  exists $y \in {\rm supp} (\mu)$ with $r(y,z)>0$.

On the set $V_\mu$ we define the equivalence relation $y \sim  z$ as follows: $y\sim  z$ if and only if
in $(V_\mu,E_\mu)$ there exists an oriented path  from $y$ to $z$ as well as an oriented path  from $z$ to $y$.
We  call $ ( V_\mu^{(\ell)}) _{\ell \in \mc L
}$ the equivalence classes of $V_\mu$ under the relation $\sim$, and set $$E^{(\ell)}_\mu:= \{ (y,z) \in E_\mu\,:\, y,z \in V^{(\ell)}_\mu \}\,.$$ Above $\mc L$ is the index set of the equivalence classes, given by $\mc L=\bb N= \{1,2, \dots \}$ if there are infinite classes, or $\mc L = \{1,2, \dots, |\mc L|\}$ if there is a finite number of classes.

\smallskip
Given $Q \in L^1_+(E)$ and given $y\not =z$ in $V$ we set $Q(y,z)=0$ if $(y,z) \not \in E$.
The support of $Q$,  denoted by $E(Q)$, is defined as $E(Q)=\{(y,z) \in E\,:\, Q(y,z)>0\}$.

\smallskip

\begin{lemma}\label{verza}
Let $Q \in L^1_+(E)$ satisfy $I(\mu,Q)<+\infty$. Then $Q(y,z)=0$ if $y \not \sim z$ in
   $(V_\mu, E_\mu)$. In particular, $E(Q) \subset \cup_{\ell \in \mc L} E^{(\ell)}_\mu$.
\end{lemma}
\begin{proof}
Suppose  that $Q(y,z)>0$.
Since $I(\mu,Q) < +\infty$ the flow $Q$ must be divergence--free. In addition, since $\Phi (q,p)=+\infty$ if $q>0$ and $p=0$, the flow $Q$ must have support contained in  $E_\mu$. By the cyclic decomposition, we can write $Q= \sum _{C \in \mc C} \hat Q(C) \bb I_C $.  In particular, $(y,z) \in C_0$ for some $C_0 \in \mc C$ with $\hat Q(C_0)>0$.  Since $Q$ has support
contained in $E_\mu$ it must be $(u,v) \in E_\mu$ for all $(u,v) \in C_0$. The  cycle $C_0$ can be divided in two oriented paths,  one from $y$ to $z$ and one from $z$ to $y$ in $(V_\mu, E_\mu)$. This implies that $y,z$ belong to the same equivalence class $V_\mu^{(\ell)}$. Since $E(Q) \subset E_\mu$ it must be  $(y,z) \in E_\mu$ and therefore $(y,z) \in E_\mu^{(\ell)}$.
\end{proof}


\subsection{Proof of Theorem \ref{teoteo}}

We define the functions  $I_1,J_1\colon \mc P(V)\to [0,+\infty]$ as the r.h.s. of \eqref{rosetta} and \eqref{freddo}, respectively:
\begin{equation}
  \label{defi-arpa1}
I_1(\mu) := \inf\big\{ I(\mu,Q)\,:\, Q\in L^1_+(E) \big\} \,,
\end{equation} and
\begin{equation}\label{defi-arpa2}
    J_1(\mu) := \left\{
 \begin{array}{ll}
    \sup\big\{ - \langle \mu,  e^{-g} L e^{g} \rangle \,:\,
  g\in L^\infty(V) \big\}\,, & \textrm{if}\ \langle \mu, r\rangle<+\infty\\
  +\infty & \textrm{otherwise}\,.
 \end{array}
 \right.
\end{equation}
As already explained we only need to prove  the equality $I_1(\mu)=J_1 (\mu)$.
By Remark \ref{silente} we can restrict to probability measures $\mu$ such that  $\langle \mu, r\rangle<+\infty$.

\smallskip

We first show  the inequality
$ I_1(\mu) \ge J_1(\mu)$. Given $g :V \to \bb R$, we define the gradient  $\nabla g : E \to \bb R$ as $\nabla g (y,z):= g(z)-g(y)$.
We now  observe that for any $g\in L^\infty(V)$ and for any divergence--free flow
$Q\in L_+^1(E)$   the following integration by parts formula is satisfied:
\begin{equation}\label{bypartscicli}
\begin{split}
\langle Q,\nabla g\rangle\,& =\,\sum_{(y,z)\in E}Q(y,z)(g(z)-g(y)) \\
& \,=\,
\sum _{y \in V} g(y) \bigl( \sum _{z:(z,y)\in E } Q(z,y)  - \sum _{z:(y,z)\in E } Q(y,z) \bigr) \,=\,-\langle \div  Q, g \rangle=0\,.
\end{split}
\end{equation}

After these observations the conclusion is simple.
We can   restrict the infimum in \eqref{defi-arpa1} to  divergence-free $Q$'s.
Fix $\phi,g\in L^\infty(V)$. We can use the variational characterization \eqref{cvi} of the rate
function $I$  and  deduce
\begin{eqnarray}
& &I(\mu,Q) 
\ge I_{\phi,\nabla g}(\mu,Q)
\nonumber\\
& & =\sum_{(y,z)\in E} \big\{
 - \mu(y) r(y,z) \big[e^{g(z)-g(y)}
  -1\big]+Q(y,z) [g(z) -g(y)]\big\} \nonumber \\
  & & = - \langle \mu, e^{-g} L e^{g} \rangle\,.\label{portonaccio}
\end{eqnarray}
In both the above  identities we used that   $Q$ is divergence-free. Minimizing over $Q$ and maximizing over $g$
in \eqref{portonaccio} we obtain that $I_1(\mu ) \geq J_1(\mu)$.

\smallskip

We show now the converse inequality $I_1 (\mu) \leq J_1 (\mu)$. As already observed after Theorem \ref{teoteo}
  the infimum defining $I_1$ in
\eqref{defi-arpa1} is attained as some flow $Q^*$, i.e.\  $I(\mu,Q^*)=I_1(\mu)$.  Taking $Q\equiv 0$ in
\eqref{defi-arpa1} we get
\begin{equation}\label{porchetta} I_1(\mu) = I(\mu,Q^*) \leq I(\mu, 0)= \langle \mu, r \rangle<+\infty \,.
\end{equation}
Due to  Lemma \ref{verza} we can write
\begin{equation}
I(\mu,Q^*)\,=\,\sum_{\ell\in \mc L } \sum_{(y,z)\in E^{(\ell) }_\mu }\Phi\Big(Q^*(y,z),\mu(y)r(y,z)\Big)
\,+\,\sum_{(y,z) \in E_\mu \setminus \cup _{\ell \in \mc L} E_\mu^{(\ell)}  }
\mu(y)r(y,z)\,.\label{deveaverlo}
\end{equation}
The next step is to show that $Q^*(y,z)>0$ for any $(y,z)\in E^{(\ell)}_\mu $. Suppose by contradiction that there exists
an edge $(y,z)\in E^{(\ell)}_\mu$ such that $Q^*(y,z)=0$. Take a  cycle  $C$ contained
in $(V^{\ell} _\mu,E_\mu^{\ell} )$ and containing the edge $(y,z)$ (it exists by definition of the equivalence relation $\sim$).
For any $\alpha \geq 0$ consider the perturbed flow $Q^*_\alpha:= Q^*+\alpha \id_C$. Then $I(\mu,Q^*_\alpha)<+\infty$. Moreover, the map $\bb R_+ \ni \alpha\to I(\mu, Q^*_\alpha) \in [0,+\infty)$  is continuous and $C^1$ on $(0,+\infty)$. Its derivative on $(0,+\infty)$  is given by
\begin{equation}\label{ninjago}
\frac{d}{d \alpha}\Big[I(\mu, Q^*_\alpha )\Big]=\sum_{(v,w)\in C}\log \frac{Q^*(v,w)+\alpha}{\mu(v)r(v,w)}\,.
\end{equation}
The  above derivative  becomes strictly negative for $\alpha$ small enough and this contradicts the fact that
$Q^*$ is a global minimizer.

Consider now an arbitrary cycle $C$ contained
in $(V^{(\ell)}_\mu ,E^{(\ell)}_\mu)$. We have just proved that $Q^*$ is strictly positive on the
edges of $C$. Hence, now we can conclude that  the above function $\bb R_+ \ni \alpha\to I(\mu, Q^*_\alpha) \in [0,+\infty)$   is   $C^1$ on all $\bb R_+$ (zero included) where its derivative is given by  \eqref{ninjago}.   Since $Q^*$ is a global minimizer the value $\alpha=0$
is a local minimum and consequently the value of the derivative in correspondence
of  $\alpha=0$ must be  zero. From
\eqref{ninjago} we get
\begin{equation}\label{rangers}
\sum_{(v,w)\in C}\log \frac{Q^*(v,w)}{\mu(v)r(v,w)}=0\,.
\end{equation}
The validity of \eqref{rangers} for any cycle $C$ contained in
$(V^{(\ell)}_\mu,E^{(\ell)}_\mu)$ implies that there exists a function $g_\ell: V^{(\ell)}_\mu \to \mathbb R$
such that
\begin{equation}\label{devbgn}
\log \frac{Q^*(y,z)}{\mu(y)r(y,z)}=g_\ell(z)-g_\ell(y)\,, \qquad \forall (y,z)\in  E^{(\ell)}_\mu\,.
\end{equation}
The function $g_\ell$ is determined up to an arbitrary additive constant
in the following way. Let $y^*$ be an arbitrary fixed element of $V^{(\ell)}_\mu$
and set  $g_\ell(y^*):=0$. For any $z\in V^{(\ell)}_\mu$ consider an arbitrary
oriented path $(z_1,\dots ,z_n)$ going from $y^*$ to $z$ in $(V_\mu, E_\mu )$  and define
\begin{equation}\label{defgl}
g_\ell(z):=\sum_{i=1}^{n-1}\log \frac{Q^*(z_i,z_{i+1})}{\mu(z_i)r(z_i,z_{i+1})}\,.
\end{equation}
If we prove that $g_\ell$ is well defined, i.e.\ that the  above
 definition
\eqref{defgl} does not depend on the  chosen path, then it is immediate to check   \eqref{devbgn}.

To show that the definition is well posed, fix  an oriented path $(u_1, \dots, u_k)$ from $z$ to $y^*$ in $(V_\mu, E_\mu )$ (it exists since $y^* \sim z$). Then  $C=(z_1, \dots, z_n, u_2, \dots , u_k)$ is a cycle going through $y^*,z$. It is trivial to check that all points in $C$ are $\sim$--equivalent to $y^*,z$, hence $C$ is a cycle in $( V^{(\ell)}_\mu, E^{(\ell)}_\mu)$.
Applying \eqref{rangers} we get
\begin{equation}\label{nido100}\sum_{i=1}^{n-1}\log \frac{Q^*(z_i,z_{i+1})}{\mu(z_i)r(z_i,z_{i+1})}=-
\sum_{j=1}^{k-1}\log \frac{Q^*(u_j,u_{j+1})}{\mu(u_j)r(u_j,u_{j+1})}\,.
\end{equation}
This shows that the l.h.s. does not depend on the particular oriented path $(z_1,\dots ,z_n)$ from $y^*$ to $z$, since the r.h.s. is path--independent.
Hence
 $g_\ell$ is well defined.

\medskip

The function $g_\ell$ does not necessarily belong to $L^\infty(V^{(\ell)}_\mu)$,
nevertheless we can improve a result similar to   \eqref{bypartscicli}:

\begin{claim}\label{eunclaim}
The series $\sum_{(y,z)\in E_\mu ^{(\ell)} }Q^*(y,z)(g_\ell(z)-g_\ell(y))$ is absolutely convergent and moreover
\begin{equation}\label{quante}
\sum_{(y,z)\in E_\mu ^{(\ell)} }Q^*(y,z)(g_\ell(z)-g_\ell(y))=0\,.
\end{equation}
\end{claim}
Since the series is absolutely convergent the l.h.s. of \eqref{quante} does not depend on the
order of summation and therefore is
 well defined.
 \begin{proof}[Proof of the claim]
 By the triangle inequality we have
 \begin{equation}\label{enzo}
q \left | \log (q/p) \right| \leq \Phi (q,p) + |q-p| \,,\qquad  q,p >0\,.
\end{equation}
Using the inequality \eqref{enzo} we have
\begin{eqnarray}
&&\sum_{(y,z)\in E_\mu^{(\ell)} }Q^*(y,z)\left|g_\ell(z)-g_\ell(y)\right| \nonumber \\
&\leq &
\sum_{(y,z)\in E(Q^*)\subset E_\mu }Q^*(y,z)\left|\log\frac{Q^*(y,z)}{\mu(y)r(y,z)}\right| \nonumber \\
&\leq & I(\mu,Q^*)+\|Q^*-Q^\mu \|_1<+\infty\,.\label{fame}
\end{eqnarray}
Above we have used \eqref{porchetta} and the fact   that $Q^\mu \in L^1_+(E)$ since   $\langle \mu, r \rangle < +\infty$.
This proves that the series  in the l.h.s. of \eqref{quante} is absolutely convergent. Trivially, using  the cyclic decomposition $Q^* = \sum_{C \in  \mc C} \hat Q(C) \id_C$, this is equivalent to the bound
\begin{equation}\label{mammina}
\sum _{(y,z) \in E_\mu ^{(\ell)}}
        \sum _{\substack{ C \in \mc C:\\ C \text{ is inside } (V^{(\ell)}_\mu, E^{(\ell)}_\mu) }  }
         \hat Q(C)  |g_\ell(z)-g_\ell(y)|
           <+\infty\,.
\end{equation}
 Due to \eqref{mammina} and the
properties of absolutely convergent series (in particular, their invariance under permutation of the addenda) we get
\begin{multline}\label{pizzona}
 \sum _{\substack{ C \in \mc C:\\ C \text{ is inside } (V^{(\ell)}_\mu, E^{(\ell)}_\mu ) } }
 \Big(
  \sum _{(y,z) \in  C} \hat Q(C)  (g_\ell(z)-g_\ell(y))\Big)\\
  \,=\,
       \sum _{(y,z) \in E_\mu ^{(\ell)}} \Big(
        \sum _{\substack{ C \in \mc C\,: \; (y,z) \in C\,,  \\ C \text{ is inside } (V^{(\ell)}, E^{(\ell)}) } }
         \hat Q(C)  (g_\ell(z)-g_\ell(y))
           \Big)\,.
 \end{multline}
On the other hand, the l.h.s. of \eqref{pizzona} is trivially zero since each sum inside the brackets is zero. The r.h.s. of \eqref{pizzona} is simply the l.h.s. of \eqref{quante} (recall Lemma \ref{verza}), thus concluding the proof of our claim.
\end{proof}

Using \eqref{devbgn} and  \eqref{quante} we obtain
\begin{equation}\label{ridotta}
\sum_{(y,z)\in E_\mu ^{(\ell)}} \Phi\Big(Q^*(y,z),\mu(y)r(y,z)\Big)=\sum_{(y,z)\in E^{(\ell)}_\mu}\mu(y)r(y,z)\left(1-e^{\nabla g_\ell (y,z)}\right)\,.
\end{equation}

\smallskip

Recall that $\mc L$ is the index set of the equivalence classes for $\sim$ in $(V_\mu,E_\mu)$. We then consider  the   oriented graph $(\mc L , \mc E)$ where the oriented edges  are given by the pairs $(\ell, \ell')$ for which
there exists an edge $(y,z)\in E_\mu$ such that $y\in V_\mu^{(\ell)} $ and $z\in V_\mu^{(\ell')}$. Then the graph $(\mathcal L,\mathcal E)$
is   an oriented acyclic graph, i.e. it contains no  cycles.

\smallskip
We can now conclude the proof. First we consider the case when $|\mc L|<+\infty$. Then by Proposition
1.4.3 in \cite{BJ} the finite acyclic oriented graph $(\mathcal L,\mathcal E)$ admits
an acyclic ordering of the vertices. This means that there exists a bijection $\hat h: \mc L \to \mc L$  such that $\hat h(\ell)<\hat h(\ell')$ for any
$(\ell,\ell')\in \mathcal E$. Then we define $h(\ell):= |\mc L| - \hat h(\ell) +1 $ to get
a   bijection $h: \mc L \to \mc L$  such that $h(\ell)>h(\ell')$ for any
$(\ell,\ell')\in \mathcal E$.

\medskip

 Consider the sequence of functions $g^{(n)}\in L^\infty(V)\,, n\in \mathbb N$ defined by
\begin{equation}\label{quasi-asilo}
g^{(n)}(y):=
\begin{cases}
g_\ell^{(n)}(y)+h(\ell)n &\; \text{ if } y\in V^{(\ell)}_\mu  \text{ for some } \ell \in \mc L\,,\\
0 & \; \text{ otherwise}\,,
\end{cases}
\end{equation}
where
\begin{equation}\label{de-paperis}
g_\ell^{(n)}(y):=
\begin{cases}
g_\ell(y) &\; \text{ if} \ |g_\ell(y)|\leq \frac n 3\\
\frac{g_\ell(y)}{|g_\ell(y)|}\frac n 3 &\;  \text{ otherwise}\,.
\end{cases}
\end{equation}
 We finally get
\begin{eqnarray}
&& J_1(\mu)\geq\lim_{n\to +\infty}-\langle \mu , e^{-g^{(n)}}Le^{g^{(n)}}\rangle \nonumber \\
& = &
  \lim_{n\to +\infty} \sum _{(y,z)\in E_\mu} \mu(y) r(y,z) \bigl( 1-e^{\nabla g ^{(n)} (y,z)}\bigr) \nonumber \\
&=&
\lim_{n\to +\infty}\Big[\sum_{\ell\in \mc L} \sum_{(y,z)\in E_\mu^{(\ell)} }\mu(y)r(y,z)\left(1-\-e^{\nabla g^{(n)}_\ell(y,z)}\right) \nonumber \\
 &+&\sum_{\ell\neq \ell'}\sum_{\substack{(y,z) \in E_\mu: \\ y \in V^{(\ell)} _\mu,\; z \in V^{(\ell')}_\mu} }
 \mu(y)r(y,z)\left(1-\-e^{ g^{(n)}_{\ell'}(y)-g^{(n)}_\ell(z)+[h(\ell')-h(\ell)]n}\right)\Big]
 \,.\label{finne}
\end{eqnarray}The above limit can be computed  applying the Dominated Convergence Theorem. To this aim we first observe that
\begin{equation}\label{santana}
\begin{cases}\bigl|\nabla g^{(n)}_\ell(y,z)\bigr|\,\leq \,\bigl|\nabla g_\ell(y,z)\bigr| &
\;  \text{  if } (y,z) \in E^{(\ell)}_\mu\,\\
 {\rm sign} \{  \nabla g^{(n)}_\ell(y,z)\} \,=\, {\rm sign} \{  \nabla g_\ell(y,z)\} &
\;  \text{  if } (y,z) \in E^{(\ell)}_\mu\,\\
g^{(n)}_{\ell'}(y)-g^{(n)}_\ell(z)+[h(\ell')-h(\ell)]n\,\leq \, -\frac 13 n
&  \; \text{ if } (y,z) \in E_\mu\,,\; y \in V^{(\ell)} _\mu,\; z \in V^{(\ell')}_\mu
\,.
\end{cases}
\end{equation}
Note that in the second case we have used that $(\ell, \ell') \in \mc E$ thus implying that $h(\ell') - h(\ell) \leq -1$. Note moreover that due to  \eqref{santana} we can write
$$  \bigl|1-\-e^{\nabla g^{(n)}_\ell(y,z)}\bigr|  \leq 1+ e^{\nabla g_\ell(y,z)}
 \,, \qquad (y,z) \in E^{(\ell)}_\mu \,.$$
Since $\langle \mu, r \rangle <+\infty$ and  due to \eqref{devbgn}  we are allowed to apply the
 the Dominated Convergence Theorem.  As a consequence, we get
\begin{equation}
\begin{split}
J_1 (\mu) & \geq \text{ r.h.s. of \eqref{finne} } \\ & = \sum_{\ell\in \mc L} \sum_{(y,z)\in E_\mu^{(\ell)} }\mu(y)r(y,z)\left(1-\-e^{\nabla g_\ell(y,z)}\right)
 +\sum_{ (y,z) \in E_\mu \setminus \cup _{\ell \in \mc L} E_\mu^{(\ell) } }
 \mu(y)r(y,z)\\
 & = I( \mu , Q^*)= I_1 (\mu)\,.
 \end{split}
\end{equation}
Note that the second equality is a byproduct of \eqref{deveaverlo} and \eqref{ridotta}.
This ends the proof of $I_1(\mu)=J_1(\mu)$  when $|\mc L|<+\infty$.
A special case with $|\mathcal L|<+\infty$ is when $\textrm{supp}(\mu)=V$. In this case $|\mathcal
L|=1$
since $E_\mu=E$ and the Markov chain $\xi$ is irreducible.

\medskip

We can now treat the general case.
  Let $\mu$ be an arbitrary probability measure
on $V$. We want to show that $J_1(\mu) \geq I_1(\mu)$.

We first observe that  $J_1(\pi)=0$, where
we recall that $\pi$ is the unique invariant measure of the Markov chain $\xi$. Indeed, since $1-e^x \leq x$ for all $x \in \bb R$,
we have for any $g\in L^\infty(V)$
\begin{equation}\label{jensen-bravo}
\sum_{(y,z)\in E}\pi(y)r(y,z)(1-e^{\nabla g(y,z)})\leq \sum_{(y,z)\in E}\pi(y)r(y,z)\nabla g(y,z)=0\,.
\end{equation}
The last equality in \eqref{jensen-bravo} follows by \eqref{bypartscicli} since $Q^\pi$ is a divergence-free element of $L^1_+(E)$. Equation \eqref{jensen-bravo} gives $J_1(\pi)\leq 0$ and the converse
inequality is obtained selecting in \eqref{defi-arpa2} a constant function $g$. This concludes the proof that $J_1(\pi)=0$.

 Since $\textrm{supp}(\pi)=V$ for any $c\in (0,1)$ $\textrm{supp}(c\mu+(1-c)\pi)=V$
so that from the result obtained in the case $|\mc L|<+\infty$ we know that
\begin{equation}\label{fac-fac}
J_1(c\mu+(1-c)\pi)=I_1(c\mu+(1-c)\pi)\,.
\end{equation}
Since $J_1$ is defined as a supremum of convex functions it is a convex function, hence
$J_1(c\mu+(1-c)\pi) \leq c J_1(\mu)+(1-c) J_1(\pi)=c J_1(\mu)$. Invoking \eqref{fac-fac}
we get
\begin{equation}\label{the-end}
J_1(\mu)\geq \liminf_{c\to 1} \frac{I_1(c\mu+(1-c)\pi)}{c}\geq I_1(\mu)\,.
\end{equation}
In the last inequality we have used the lower semicontinuity of $I_1$.

\section{Proof of Proposition \ref{simmetria}}\label{demo_simm}

We will use both the equivalent representations \eqref{rosetta} and \eqref{freddo} of the rate
function $\mathcal I$.  We take $\mu \in \mathcal{P}(V)$ with $\langle \mu, r \rangle < +\infty$  and
recall equation \eqref{extended} that defines the
extended function $g:V \to [-\infty,+\infty)$  as $g(y):=\log \sqrt{\mu(y)/\pi(y)}$, with the convention $\log 0:=-\infty$. We consider also the sequence of functions $g^{(n)}\in L^\infty(V)$ defined by
$$
g^{(n)}(y):=\left\{
\begin{array}{ll}
g(y)& \textrm{if}\ |g(y)|\leq n\,,\\
\frac{g(y)}{|g(y)|}n & \textrm{if}\ |g(y)|>n\,,
\end{array}
\right.
$$
where by $(-\infty)/(+\infty)$ we mean $-1$. We observe that (use $2ab \leq a^2+b^2$)
$$ \sum _{y \in V} \sum_{z \in V} \mu(y) r(y,z) e^{\nabla g (y,z) }= \sum _y \sum_z \sqrt{ \mu(y) r(y,z) \mu(z) r(z,y) }  \leq \langle \mu, r \rangle < +\infty.$$
The above bound and the inequality $|1-e^{\nabla g^{(n)} (y,z) }|\leq 1+ e^{\nabla g (y,z) }$, allows to apply the   Dominated Convergence Theorem. As a consequence, by
 an elementary computation we get
\begin{equation}
\lim_{n\to +\infty}-\langle \mu, e^{-g^{(n)}}Le^{g^{(n)}}\rangle
=\sum_{\left\{y,z\right\}\in E^u}\left(\sqrt{\mu(y)r(y,z)}-\sqrt{\mu(z)r(z,y)}\right)^2\,,
\label{lower}
\end{equation}
where with $E^u$ we denote the set of unordered bonds. Since we are considering the reversible case then necessarily if $(y,z)\in E$ then also $(z,y)\in E$. Consequently $\left\{y,z\right\}\in E^u$ when both $(y,z)$
and $(z,y)$ belong to $E$. Equation \eqref{lower} implies that $\mathcal I(\mu)$ is greater or equal to the right hand side of \eqref{lower}.

To prove the converse inequality we restrict the infimum in \eqref{rosetta} to symmetric
flows, i.e. to flows $Q$ such that $Q(y,z)=Q(z,y)$ for any $(y,z)\in E$. Then we can define
$S:E^u\to \mathbb R^+$ by $S(\left\{y,z\right\}):=Q(y,z)=Q(z,y)$. For symmetric flows
the rate function $I(\mu,Q)$ can be written as
\begin{equation}\label{symm}
\begin{split}
\sum_{\{y,z\}\in E^u}\Big[2S(\{y,z\}) & \log\frac{S(\{y,z\})}{\sqrt{\mu(y)r(y,z)\mu(z)r(z,y)}}
\\ &+\mu(y)r(y,z)+\mu(z)r(z,y)-2S(\{y,z\})\Big]\,.
\end{split}
\end{equation}
The minimization procedure in \eqref{symm} is easy since the zero divergence constraint is always satisfied.
We can then solve an independent variational problem for each unordered bond,  without any constraint apart
the non-negativity of $S$. On the bond $\{y,z\}$ the minimizer is
$$
S^*(\{y,z\})=\sqrt{\mu(y)r(y,z)\mu(z)r(z,y)}\,.
$$
Calling $Q^*$ the associated symmetric flow we get that $I(\mu,Q^*)$ coincides
with the right hand side of \eqref{lower}. 
 This completes the proof.

\section{Alternative proof of Theorem \ref{teoteo} for $V$ finite}\label{rocky}

As explained after Theorem \ref{teoteo} we only need to show the identity  $I_1=J_1$  (recall \eqref{defi-arpa1} and \eqref{defi-arpa2}).

\smallskip

In the finite case an interesting proof of this result is obtained observing that it is a special case
of the Fenchel-Rockafellar Theorem (see for example \cite{Bre}). In the case $|V|=+\infty$ this strategy
does not work since the continuity requirement in the following general statement is missing.
Consider a topological vector space  $X$ and its dual $X^*$.  Let
$\phi,\psi: X \to (-\infty,+\infty]$  be  two proper (i.e. not identically equal to $+\infty$) extended convex functions   such that $\phi+\psi$ is proper and there exists an $x_0\in X$ where either $\phi(x_0)<+\infty$ and $\phi$ is continuous at $x_0$ or $\psi(x_0)<+\infty$ and $\psi$ is continuous at $x_0$. The Fenchel-Rockafellar
Theorem states that
\begin{equation}
\inf_{x\in X}\left\{\phi(x)+\psi(x)\right\}=\sup_{f\in
X^*}\left\{-\phi^*(-f)-\psi^*(f)\right\}\,, \label{FeRo}
\end{equation}
where, given $\gamma: X \to \mathbb (-\infty,+ \infty]$,
$$
\gamma^*(f):=\sup_{x\in X}\left\{\langle f,
x\rangle-\gamma(x)\right\}\,, \qquad f\in X^*\,,
$$
is the Legendre transform of $\gamma$.

\smallskip

Fix $\mu \in \mathcal{P}(V)$ and recall the definition of the graph $(V_\mu, E_\mu)$ given at the beginning of Section \ref{demo}, as well as  the  equivalence relation $\sim$ leading to the equivalence classes $V^{(\ell)}_\mu$, $\ell \in \mc L$, and associated set of edges $E^{(\ell)}_\mu$. Since $V$ is finite the condition $\langle \mu, r \rangle < +\infty$ is automatically satisfied

We want to apply the  Fenchel-Rockafellar
Theorem with   $X=L^1(E_\mu )$ endowed of the standard $L^1$--norm and
$X^*=L^\infty(E_\mu)$. Clearly in the finite case we could work with the simpler choice $X=X^*=\mathbb R^{E_\mu}$
with the Euclidean topology (we use instead a more general notation having
in mind some possible extensions to the infinite case). Our choice for the functions $\phi, \psi$ is
\begin{align*}
& \phi(Q):=
\begin{cases}
    \displaystyle{
    \sum_{(y,z)\in E_\mu} \Phi \big( Q(y,z),Q^\mu(y,z) \big)
    }& \textrm{if  } \;  Q\in L^1_+(E_\mu)\,,
    \\
    \; +\infty  & \textrm{otherwise}.
  \end{cases}
\\
&
\psi(Q):=
\begin{cases}
0 & \textrm{if} \; \div Q=0\,,\\
\; +\infty & \textrm{otherwise\,.}
\end{cases}
\end{align*}
Given $Q \in L^1(E_\mu)$ the divergence $\div Q: V \to \bb R$ is still defined as
$$ \div Q(y)= \sum_{(y,z) \in E_\mu} Q(y,z)- \sum_{(z,y) \in E_\mu} Q(z,y)\,.
$$

The above functions $\phi, \psi$ are proper convex functions (recall that $\Phi(\cdot, p)$ is convex for any $p \geq 0$).  Moreover, since $\Phi(\cdot, p)$ is a continuous function on $(0,+\infty)$ for all $p>0$,  we conclude that $\phi(Q)<+\infty$ and $\phi$ is continuous at $Q$ for any $Q \in L^1(E_\mu)$ such that
$Q(y,z) >0$ for all $(y,z) \in E_\mu$. Finally, we note that the function $\phi+\psi$ is proper since finite on the zero flow.  Note that working with $E$ instead  of $E_\mu$, neither $\psi$ nor $\phi$ would have  satisfied the condition of boundedness and continuity in at least one point.

Clearly it holds 
\begin{equation}\label{camillo}
I_1 (\mu)= \inf_{Q\in
L^1_+(E)}I(\mu,Q)  = \inf_{Q\in L^1(E_\mu )}\left\{\phi(Q)+\psi(Q)\right\} \,.\end{equation}
The fact that $I_1(\mu)=J_1(\mu)$ then follows as a byproduct of the  Fenchel-Rockafellar
Theorem and the following Claims \ref{claim1}  and \ref{claim2}.
\begin{claim}\label{claim1}
For any  $f \in L^\infty (E_\mu)$ we have
\begin{equation}
\phi^*(f)=\sum_{(y,z)\in
E_\mu }\mu(y)r(y,z)\left(e^{f(y,z)}-1\right)\,,\qquad f\in L^\infty(E_\mu)\,.
\end{equation}
\end{claim}
\begin{proof}
For each edge $(y,z)\in E_\mu$, the function $\bb R _+ \ni u \to f(y,z) u - \Phi(u, Q^\mu(y,z) )$  has maximum value $\mu(y)r(y,z)\left(e^{f(y,z)}-1\right)$    attained at $u = e^{f(y,z)} Q^\mu (y,z)=e^{f(y,z) }\mu(y) r(y,z)$.
\end{proof}

Before stating Claim \ref{claim2} we prove  a geometric characterization of gradient functions on
oriented graphs. To this aim we fix some language.

We define $E_\mu^*$ as the set of oriented edges $(y,z)$ such that $(y,z) \in E_\mu$ or $(z,y) \in E_\mu$. Given a function $f$ on $E_\mu$ we can extend it to a function $f_*$ on  $E_\mu^*$  setting
$$ f_*(y,z):=\begin{cases}
\;\; f(y,z) & \text{ if } (y,z) \in E_\mu\,,\\
-f(z,y) & \text{ if } (y,z) \in E^*_\mu \setminus  E_\mu\,.
\end{cases}
$$
We say that $(x_1, x_2, \dots,x_n)$ is a generalized path from $x_1$ to $x_n$ in $(V_\mu, E_\mu)$ if for any $i=1,\dots, n-1$ it holds $(x_i,x_{i+1})\in E^*_\mu$ (in other words,  $(x_1, x_2, \dots,x_n)$ is an oriented path in $(V_\mu, E^*_\mu)$).
Given a generalized path $\gamma $ we define $\int _\gamma f$ as
$$ \int _\gamma f := \sum _{i=1}^{n-1} f_*(x_i,x_{i+1}) $$
if $\gamma$ is given by $x_1, x_2, \dots,x_n$.

\begin{lemma}\label{calduccio}
Given $f\in
L^\infty(E_\mu)$ there exists a $g:V\to \mathbb R$ such that
$f=\nabla g $, i.e. \ such that $f(y,z)= g(z)-g(y)$ for all $(y,z) \in E_\mu$, if and only if
for any  pair of   generalized paths $\gamma$ and $\gamma'$ having the same initial point and the same final point it holds
\begin{equation}
\label{ciocco}
 \int _\gamma  f= \int _{\gamma'} f\,.\end{equation}
\end{lemma}
\begin{proof} If is simple to check that
if $f= \nabla g$  then $f_*(y,z)= g(z)-g(y)$. This implies that $\int _\gamma f= g(z)-g(y)$ for any generalized path $\gamma$ in  $(V_\mu, E_\mu)$ from $y$ to $z$. Therefore \eqref{ciocco} is satisfied whenever $\gamma,\gamma'$ have the same extremes.

Suppose on the other hand that \eqref{ciocco} is satisfied for any $\gamma, \gamma'$ having the same extremes. We introduce on $V_\mu$ the equivalent relation $\sim^*$ saying that $y\sim^* z
$ if there exists a generalized path from $y$ to $z$ in  $(V_\mu, E_\mu)$. It is simple to check that we have indeed an equivalence relation.  For each equivalence class $W \subset V_\mu$ we fix a reference site $x_* \in W$ and define $g$ on $W$ setting $g(y):= \int _\gamma f$  where $\gamma$ is any generalized path from $x_*$ to $y$. Due to \eqref{ciocco} the definition is well posed. Let us check that $\nabla g =f$. To this aim we fix $(y,z)\in E_\mu$. Clearly this implies $y\sim_* z$. Call $x_*$  the reference site of their equivalence class and fix a path $\gamma $ from $x_*$ to $y$. If $\gamma=(x_1, x_2, \dots ,x_n)$ set $\tilde \gamma:= (x_1, x_2, \dots, x_n, z)$. The path $\tilde \gamma$ is a generalized path from $x_*$ to $z$.  By definition $g(z) = \int _{\tilde \gamma} f= f(y,z) +\int _{\gamma } f= f(y,z)+g(y)$, thus concluding the proof.
\end{proof}

We can now state our final claim:

\begin{claim}\label{claim2}
For any  $f \in L^\infty (E_\mu)$ we have \begin{equation}
\psi^*(f)=
\begin{cases}
0 & \textrm{if} \; f= \nabla g \text{ for some $g:V \to \mathbb R$}\,,\\
+\infty & \textrm{otherwise}\,.
\end{cases}
\label{trota}
\end{equation}
\end{claim}
\begin{proof} By a simple integration by parts it is trivial to check that it holds $\langle f,Q\rangle=0$ if    $f\in L^\infty(E_\mu)$ is  of gradient type (i.e $f=\nabla g$) and $Q\in L^1(E_\mu)$ is a divergence--free flow (see \eqref{bypartscicli}).
As a consequence  we get
$$
\psi^*(f)=\sup_{Q\in L^1(E_\mu)}\left\{\langle f,Q\rangle-\psi(Q)\right\}=0\,,
$$
for any $f$ of gradient type.

Conversely suppose that $f$ is not of gradient type. Then, by Lemma \ref{calduccio}
there exist two generalized paths $(x_1,x_2, \dots ,x_n)$ and $(y_1,y_2, \dots,y_m)$ in $(V_\mu,E_\mu)$ such that $x_1=y_1$, $x_n=y_m$  and
\begin{equation}\label{nonna}
\sum_{i=1}^{n-1}f_*(x_i,x_{i+1})-\sum_{j=1}^{m-1}f_*(y_j,y_{j+1})>0\,.
\end{equation}
The given $\lambda >0$ we define the divergence--free  $Q_\lambda \in L^1(E_\mu)$ as
$$
Q_\lambda(y,z):=
\begin{cases}
\lambda & \textrm{if} \; (y,z)=(x_i,x_{i+1})\; i=1,\dots,n-1\,,\\
-\lambda & \textrm{if} \; (y,z)=(x_{i+1},x_{i})  \text{ and } (x_i,x_{i+1}) \not \in E_\mu \; i=1,\dots,n-1\\
-\lambda & \textrm{if} \; (y,z)=(y_j,y_{j+1})\; j=1,\dots,m-1\,,\\
\lambda & \textrm{if} \; (y,z)=(y_{j+1},y_j) \text{ and } (y_j,y_{j+1})\not \in E_\mu  \; j=1,\dots,m-1\,,\\
0& \textrm{otherwise}\,.
\end{cases}
$$
Then $\langle f,Q_\lambda
\rangle$ equals $\lambda$ times  the r.h.s. of \eqref{nonna}, thus implying that
 $\lim_{\lambda \to +\infty}\langle f,Q_\lambda
\rangle=+\infty$. In particular, we obtain
$$
\psi^*(f)\geq \lim_{\lambda \to +\infty}\left(\langle
f,Q_\lambda\rangle-\psi(Q_\lambda)\right)=+\infty\,.
$$
This ends the proof of our claim. \end{proof}

\end{document}